\documentclass[12pt,a4paper]{amsart}

\usepackage{amsmath,amsthm,amssymb}

\theoremstyle{plain}
\newtheorem{thm}{Theorem}

\newtheorem{lemma}{Lemma}
\newtheorem{proposition}{Proposition}
\newtheorem{corollary}{Corollary}

\theoremstyle{definition}
\newtheorem{definition}{Definition}

\theoremstyle{remark}

\newtheorem{example}{Example}

\newcommand{\IC}{\text{${\mathbb{C}}$}}

\newcommand{\IN}{\text{${\mathbb{N}}$}}

\begin{document}

\title[Generalized Fourier series ]{Generalized Fourier series by double trigonometric system}
\thanks{\it Math Subject Classifications:
  42A24 (Primary);  41A30, 42A65, 42A8 (Secondary)}

\

\

\author{K. S. Kazarian}
\thanks{ Dept. of Mathematics, Mod. 17, Universidad Aut\'onoma de Madrid, 28049, Madrid, Spain e-mail: kazaros.kazarian@uam.es}

\keywords{generalized Fourier series, multiplicative completeness, strong $x-$singularity} 

\begin{abstract}{Necessary and sufficient conditions are obtained on the function $M$ such that
$\{ M(x,y) e^{i kx}e^{i my}: (k,m)\in \Omega \}$ is complete and minimal in $L^{p}(\mathbb{T}^{2})$
when $\Omega^{c}=\{(0,0)\}$ and $\Omega^{c} = 0\times\mathbb{Z}$.
If $\Omega^{c} = 0\times\mathbb{Z}_{0},$ $\mathbb{Z}_{0} = \mathbb{Z}\setminus\{0\}$ it is proved that the system
$\{ M(x,y) e^{i kx}e^{i my}: (k,m)\in \Omega \}$ cannot be complete minimal in $L^{p}(\mathbb{T}^{2})$  for any
$M\in L^{p}(\mathbb{T}^{2})$.
In the case $\Omega^{c}=\{(0,0)\}$ necessary and conditions are found in terms of the one dimensional case.
}
\end{abstract}

\maketitle

\section{Introduction}

The present study is motivated by the desire to  extend the concept of generalized Fourier series (GFS) for functions of various variables. The concept of GFS can be described as follows.
Let $(X,\Sigma,\mu)$ be a measurable space with a positive measure $\mu,\mu(X)>0$ and let $L^{2}(X,\Sigma,\mu)$ be
the space of measurable functions $f:X\rightarrow \IC$ with the norm $ \|f\|_{2}= (\int_{X} |f(t)|^{2} dt)^{\frac{1}{2}} <\infty$.
For a complete orthonormal system $\Phi = \{\varphi_{k}\}_{k=1}^{\infty}$
\begin{equation}\label{1}
  \sum_{k=1}^{\infty} a_{k}(f) \varphi_{k} = f\quad \mbox{in}\quad L^{2}(X,\Sigma,\mu)
\end{equation}
where for any $k\in \IN$
\begin{equation}\label{2}
     a_{k}(f)  =  \int_{X}f(t) \overline{\varphi_{k}(t)} d\mu.
\end{equation}
The series (\ref{1})-(\ref{2}) is the Fourier  series of the function $f$ with respect to the system $\Phi$.
When the system $\Phi$ is the trigonometric system it is  called the Fourier series of the given function.
Representation of a given function by a trigonometric series is a classical topic (see \cite{Ri:1}, \cite{L:1} and many others).
It is well known that a measurable function can be represented by a $\Phi-$ series   where the coefficients are not defined by (\ref{2}).

 It seems something transcendental to find an algorithm that determines the coefficients $b_{k}, k\in \mathbb{N}$ such that the series $ \sum_{k=1}^{\infty} b_{k} \varphi_{k}$ represents a given function $f$ when $f$ is not integrable. Such a problem was formulated by N.N. Luzin \cite{L:1}.
 The following strategy  can be an inexhaustible source for the study of the Luzin's problem.

 At the first step fix a subset $\mathbb{N}_{1} \subset \mathbb{N}$ such that for some $m\in L^{2}(X,\Sigma,\mu)$ the system
\begin{equation}\label{3}
\{m\varphi_{k}\}_{k\in \mathbb{N}\setminus \mathbb{N}_{1}}\quad \mbox{ is complete in}\quad L^{2}(X,\Sigma,\mu).
\end{equation}
 Determine those functions $m$
for which $\{m\varphi_{k}\}_{k\in \mathbb{N}\setminus \mathbb{N}_{1}}$  is complete and minimal in $ L^{2}(X,\Sigma,\mu)$ if it is possible. Afterwards if we fix any such $m$ then the system $ \{m\varphi_{k}\}_{k\in \mathbb{N}\setminus \mathbb{N}_{1}}$ will have a unique biorthogonal system $ \{ \psi_{k}\}_{k\in \mathbb{N}\setminus \mathbb{N}_{1}}$ in $L^{2}(X,\Sigma,\mu)$. When $ \{ \psi_{k}\}_{k\in \mathbb{N}\setminus \mathbb{N}_{1}}$ is total with respect to the space $ L^{2}(X,\Sigma,\mu)$ then for any measurable function $g$ such that $mg \in L^{2}(X,\Sigma,\mu)$ one can consider the series
\begin{equation}\label{3:1}
  \sum_{k\in \mathbb{N}\setminus \mathbb{N}_{1}}  b_{k}(g) \varphi_{k}\qquad  b_{k}(g)= \int_{X}g(t) \overline{\psi_{k}(t)} d\mu, \quad k\in \mathbb{N}\setminus \mathbb{N}_{1}.
\end{equation}
The trigonometric system is the best object for testing the described idea because of its importance in various areas of mathematics.

  Generalized Fourier series and some applications were studied in \cite{KSK:1}--\cite{KSK:3} when $\mbox{card}\,\mathbb{N}_{1} <\infty$.  It is not known if the described strategy  is viable for the trigonometric system if $\mbox{card}\,\mathbb{N}_{1} = \infty$ (see \cite{KSK:4}). Any essential progress in the problem formulated below will be very helpful to clarify the question.

   We denote $\mathbb{T} = {\mathbb{R}}/{2\pi\mathbb{Z}}$ and consider  the complex form of the trigonometric system  $\{ e^{i kx}: k\in \mathbb{Z} \}  $ defined on the set $\mathbb{T}$, where the set of all integer numbers is denoted by $\mathbb{Z}$. The following theorems were proved in \cite{KSK:4}.
Let $$1\leq n_{1} < n_{2}<\cdots < n_{k} < n_{k+1} < \cdots$$
 and let 
 $$\Omega = \{-n_{k}\}_{k=1}^{\infty}\cup \{n_{k}\}_{k=1}^{\infty}$$
  be an infinite set of natural numbers such that $\Omega^{c} \neq \emptyset$, where
\[
\Omega^{c} = \mathbb{Z}\setminus \Omega = \{-m_{k}\}_{k=1}^{\infty}\cup \{0\}\cup\{m_{k}\}_{k=1}^{\infty}.
\]
 Let $S_{\Omega} = \overline{\mbox{span}}_{L^{1}(\mathbb{T})}\{ e^{i kx} : k\in \Omega^{c} \} $. If $
 p\geq 1$ then its conjugate number $p'$ is defined by the equation $\frac{1}{p} + \frac{1}{p'} =1$.\par
{\bf Theorem A}.
Let $1\leq p <\infty$ and  let $M\in L^{p}(\mathbb{T})$.\newline %
 Then the system $\{ M(x) e^{i kx}: k\in \Omega\}  $ is complete in $L^{p}(\mathbb{T})$ if and only if the following condition holds:
\begin{equation}\label{C}
 \makebox{\parbox{4.1in}{ If
 ${g}[\overline{M}]^{-1} \in L^{p'}(\mathbb{T})$ for some $g\in S_{\Omega}$   then $g(x) = 0$ a.e. }}
\end{equation}

{\bf Theorem B}. Let $1\leq p <\infty$ and  let $M\in L^{p}(\mathbb{T})$. \newline %
 The system $\{ M(x) e^{i kx}: k\in \Omega\}  $ is minimal in $L^{p}(\mathbb{T})$ if and only if the following condition holds:
\begin{equation}\label{M}
 \makebox{\parbox{4.1in}{ If for any $k\in \Omega$ there exists a function $g_{k}\in S_{\Omega}$ such that \newline
 $[e^{i kx} + g_{k}(x) ][\overline{M}]^{-1} \in L^{p'}(\mathbb{T})$.  }}
\end{equation}
The following open problem was formulated in \cite{KSK:4}.

{\bf Problem}.
Describe pairs $(\Omega,M)$ with $\mbox{card}\,\Omega^{c} =\infty$ such that conditions (\ref{C}) and (\ref{M}) hold simultaneously.

Unfortunately no any subset $\Omega \subset \mathbb{Z}, \mbox{card}\,\Omega^{c} =\infty$ is known such that the conditions (\ref{C}) and (\ref{M}) hold simultaneously. In the present paper it is shown that the similar question for the double trigonometric system has a positive answer.
It should be mentioned that for the Haar system the described strategy can be successfully implemented when $\mbox{card}\,\mathbb{N}_{1} = \infty$ (see \cite{KSK:1}, \cite{KSK:5}). First results on multiplicative completion of sets of functions were obtained in \cite{BP}, \cite{PZ:Ann}.

\section{Multiplicative completion of some subsystems of \\ the double trigonometric system}

We will consider the double trigonometric system. The $n$-multiple case can be studied in a similar way.
We suppose that $\Omega \subset \mathbb{Z}^{2}$ is an  infinite set  such that $\Omega^{c}$ is not empty, where  $\Omega^{c} = \mathbb{Z}^{2}\setminus \Omega$. In this case we modify the definition of the class
\[
S_{\Omega} = \{f\in L^{1}(\mathbb{T}^{2}): \int_{\mathbb{T}^{2}} f(x,y) e^{-i kx}e^{-i my}dxdy =0 \quad \forall\, (k,m)\in \Omega \}.
\]
 It is clear that $S_{\Omega}$ is a closed subspace of $L^{1}(\mathbb{T}^{2})$.

\begin{thm}\label{thm2:C}
Let $1\leq p <\infty$ and  let $M\in L^{p}(\mathbb{T}^{2})$. Then the system
\begin{equation}\label{eq:sys}
\{ M(x,y) e^{i kx}e^{i my}: (k,m)\in \Omega \}
\end{equation}
is complete in $L^{p}(\mathbb{T}^{2})$ if and only if the following condition holds:
\begin{equation}\label{eq:7}
 \mbox{ If
 ${g}[\overline{M}]^{-1}\in L^{p'}(\mathbb{T}^{2})$ for some $g\in S_{\Omega},$ then
    $g(x) = 0$ a.e. }
\end{equation}
\end{thm}
\begin{proof} Suppose that (\ref{eq:sys}) is complete in $L^{p}(\mathbb{T}^{2})$ and let $g\in S_{\Omega}$  be a non trivial function such that $\frac{g}{\overline{M}} \in L^{p'}(\mathbb{T}^{2})$. Then for any $(k,m)\in \Omega$
\[
\int_{\mathbb{T}^{2}} M(x,y) e^{i kx}e^{i my} \frac{\overline{g(x,y)}}{M(x,y)} dxdy = \int_{\mathbb{T}^{2}} e^{i kx}e^{i my} \overline{g(x,y)} dxdy = 0.
\]
Which contradicts the completeness of the system (\ref{eq:sys}). Hence, (\ref{eq:7}) holds.

Now suppose that (\ref{eq:7}) holds and for some $\varphi \in L^{p'}(\mathbb{T}^{2})$
\[
\int_{\mathbb{T}^{2}} M(x,y) e^{i kx} e^{i my} \overline{\varphi(x,y)} dx dy = 0\qquad \mbox{for all}\quad (k,m)\in \Omega.
\]
Which yields that \newline 
$\overline{M}\varphi \in S_{\Omega}$ and $\varphi(x,y) =0$ a.e. on $\mathbb{T}^{2}$.
\end{proof}

\begin{thm}\label{thm2:M} Let $1\leq p <\infty$ and  let $M\in L^{p}(\mathbb{T}^{2})$. The system $(\ref{eq:sys})$ is minimal in $L^{p}(\mathbb{T}^{2})$ if and only if the following condition holds:
\[
\mbox{ If for any $(k,m)\in \Omega$ there exists a function $g_{k,m}\in S_{\Omega}$ such that}
\]
\begin{equation} \label{eq:8}
[e^{i kx}e^{i my} + g_{k,m}(x,y) ][\overline{M(x,y)}]^{-1} \in L^{p'}(\mathbb{T}^{2}).
\end{equation}
\end{thm}
\begin{proof} Suppose that (\ref{eq:sys}) is minimal in $L^{p}(\mathbb{T}^{2})$. Then there exists a system 
 $\{\varphi_{j,l} \}_{(j,l)\in \Omega} \subset L^{p'}(\mathbb{T}^{2}) $ such that
\[
\int_{\mathbb{T}^{2}} M(x,y) e^{i kx}e^{i my} \overline{\varphi_{j,l}(x,y)} dx dy = \delta_{kj} \delta_{ml}  \quad \forall (k,j), (m,l)\in \Omega.
\]
Hence, for any $(j,l)\in \Omega$ we have that
\[
\int_{\mathbb{T}^{2}}  e^{i kx}e^{i my} [M(x,y)\overline{\varphi_{j,l}(x,y)} - \frac{1}{(2\pi)^{2}}e^{-i jx} e^{-i ly}]dx dy = 0  \quad\forall\, (k,m)\in \Omega.
\]
Which yields $\overline{M(x,y)} \varphi_{j,l}(x)- \frac{1}{(2\pi)^{2}}e^{i jx}e^{i ly}= g_{j,l}(x,y) \in S_{\Omega}$. The proof of the necessity is finished.\newline
If (\ref{eq:8}) holds then it is easy to check that  $\{\varphi_{k,m} \}_{(k,m)\in \Omega} \subset L^{p'}(\mathbb{T}^{2}), $ where
\begin{equation}\label{eq:bisy}
\varphi_{k,m}(x) = \frac{e^{i kx}e^{i my} + g_{k,m}(x) }{(2\pi)^{2} \overline{M(x,y)}}\qquad \mbox{for}\quad (k,m)\in \Omega
\end{equation}
is biorthogonal to (\ref{eq:sys}).

\end{proof}

\subsection{The case $\Omega^{c}=\{(0,0)\}$}

\

Denote $\mathbb{Z}_{0} = \mathbb{Z} \setminus \{0\}$ and $\mathbb{Z}^{2}_{0} = \mathbb{Z}^{2} \setminus \{(0,0)\}$.
\begin{thm}\label{thm:0}
Let $1\leq p <\infty$ and  let $M\in L^{p}(\mathbb{T}^{2})$. Then the system
\begin{equation}\label{eq:sys10}
\{ M(x,y) e^{i kx}e^{i my}: (k,m)\in \mathbb{Z}^{2}_{0} \}
\end{equation}
is complete and minimal in $L^{p}(\mathbb{T}^{2})$ if and only if the systems $\{ u(t) e^{i nt}: n\in \mathbb{Z}_{0}\}$ and $\{ v(t) e^{i nt}: n\in \mathbb{Z}_{0} \}$  are complete and minimal in $L^{p}(\mathbb{T})$, where
\begin{equation}\label{w2}
\frac{1}{u(x)} = \left(\int_{\mathbb{T}} |M(x,y)|^{-p'} dy\right)^{\frac{1}{p'}}\quad \mbox{and}\quad \frac{1}{v(y)} = \left(\int_{\mathbb{T}} |M(x,y)|^{-p'} dx\right)^{\frac{1}{p'}}.
\end{equation}
\end{thm}
\begin{proof}
By Theorem \ref{thm2:C} it follows that the system (\ref{eq:sys10}) is complete in $L^{p}(\mathbb{T}^{2})$ if and only if
\begin{equation}\label{c0}
\int_{\mathbb{T}^{2}}|M(x,y)|^{-p'} dxdy = +\infty.
\end{equation}
Hence, by  Theorem \ref{thm2:M} the system (\ref{eq:sys10}) is  minimal  in $L^{p}(\mathbb{T}^{2})$ if there exist unique numbers $a_{kl}\in \mathbb{C}$, $(k,l)\in \mathbb{Z}_{0}$ such that
\begin{equation}\label{m0}
\int_{\mathbb{T}^{2}}|e^{ikx} e^{il y} - a_{kl} |^{p'} |M(x,y)|^{-p'} dxdy < +\infty \quad \mbox{for any}\quad (k,l)\in \mathbb{Z}_{0}.
\end{equation}
We consider (\ref{m0})  respectively for  $(k,0)$ and $(0,l)$, where $k$ and $l$ belong to $\mathbb{Z}_{0}. $ By
 the Fubini-Tonelli theorem it follows that the functions $u$ and $v$ are positive a.e. on $\mathbb{T}$. On the other hand we have that for almost any $x\in \mathbb{T}$
 \[
 2\pi \leq \left(\int_{\mathbb{T}} |M(x,y)|^{p} dy\right)^{\frac{1}{p}} \left(\int_{\mathbb{T}} |M(x,y)|^{-p'} dy\right)^{\frac{1}{p'}},
\]
which yields
\[
\int_{\mathbb{T}} u(x)^{p} dx \leq \frac{1}{(2\pi)^{p}} \int_{\mathbb{T}^{2}} |M(x,y)|^{p} dxdy.
\]
 Similarly we obtain that $v\in L^{p}(\mathbb{T})$. Afterwards by (\ref{c0}) and (\ref{m0}) we easily obtain that there exists $x_{0}\in \mathbb{T}$ such that
\begin{equation}\label{cm}
\int_{\mathbb{T}} u(x)^{-p'}  dx = +\infty \quad \mbox{and}\quad \int_{\mathbb{T}} |e^{ikx}  -  e^{ikx_{0}} |^{p'} u(x)^{-p'}  dx < +\infty.
\end{equation}
By Proposition 3 of \cite{KSK:5} it follows that  the system $\{ u(t) e^{i nt}: n\in \mathbb{Z}_{0} \}$ is complete and minimal in $L^{p}(\mathbb{T})$. Similarly we obtain that $\{ v(t) e^{i nt}: n\in \mathbb{Z}_{0} \}$ is complete and minimal in $L^{p}(\mathbb{T})$.

\end{proof}

The following theorem gives another characterization.
\begin{thm}\label{thm:1}
Let $1\leq p <\infty$ and  let $M\in L^{p}(\mathbb{T}^{2})$. Then the system $(\ref{eq:sys10})$
is complete and minimal in $L^{p}(\mathbb{T}^{2})$ if and only if  holds $(\ref{c0})$ and
\[
 \int_{\mathbb{T}^{2}}|\sin\frac{x-x_{0}}{2}|^{p'} |M(x,y)|^{-p'}dxdy< +\infty,
\]
\[
\int_{\mathbb{T}^{2}}|\sin\frac{y-y_{0}}{2}|^{p'} |M(x,y)|^{-p'}dxdy< +\infty.
\]
for some $(x_{0},y_{0})\in \mathbb{T}^{2}$.
\end{thm}
\begin{proof}
We skip the proof of the  necessity because the arguments are similar to those used in the proof of the previous theorem.
To finish the proof we have to check the relations (\ref{m0}) for $a_{kl} = e^{ikx_{0}} e^{il y_{0}} $.
Write
\begin{eqnarray*}
 &\left(\int_{\mathbb{T}^{2}} |e^{ikx} e^{il y}  - e^{ikx_{0}} e^{il y_{0}} |^{p'} |M(x,y)|^{-p'} dxdy \right)^{\frac{1}{p'}}\\
\leq &\left(\int_{\mathbb{T}^{2}} | e^{il y}  -  e^{il y_{0}}|^{p'} |M(x,y)|^{-p'} dxdy\right)^{\frac{1}{p'}}\\
+ &\left(\int_{\mathbb{T}^{2}} |e^{ikx}   - e^{ikx_{0}} |^{p'} |M(x,y)|^{-p'} dxdy\right)^{\frac{1}{p'}} < +\infty.
\end{eqnarray*}
\end{proof}

\begin{corollary}
Let $1\leq p <\infty$ and  let $M\in L^{p}(\mathbb{T}^{2})$. Then for any $(\nu, \mu)\in \mathbb{Z}^{2}$  the system
\[
\{ M(x,y) e^{i kx}e^{i my}: (k,m)\in \mathbb{Z}^{2}\setminus (\nu, \mu) \}
\]
is complete and minimal in $L^{p}(\mathbb{T}^{2})$ if and only if the system
\[
\{ M(x,y) e^{i kx}e^{i my}: (k,m)\in \mathbb{Z}^{2}_{0} \}
\]
is complete and minimal in $L^{p}(\mathbb{T}^{2})$ .
\end{corollary}
The assertion of the corollary is obvious because the multiplying the elements of the system  $\{  e^{i kx}e^{i my}: (k,m)\in \mathbb{Z}^{2}\setminus (\nu, \mu) \}$ by $e^{-i \nu x}e^{-i \mu y}$
 we obtain the system $\{ e^{i kx}e^{i my}: (k,m)\in \mathbb{Z}^{2}_{0} \}.$ On the other hand it is easy to observe that in our case the conditions (\ref{eq:7}),(\ref{eq:8}) remain true if $M(x,y)$ is multiplied by a function with modulus equal to one almost everywhere.

\begin{example}
Let $1\leq p <\infty,$ $(x_{0},y_{0})\in \mathbb{T}^{2}$ and  let 
$$ M(x,y) = |\sin\frac{x-x_{0}}{2}|^{\alpha} + |\sin\frac{y-y_{0}}{2}|^{\alpha},$$
 where $\frac{1}{p'} \leq \alpha < 1+ \frac{1}{p'}$. Then the system $\{ e^{i kx}e^{i my}: (k,m)\in \mathbb{Z}^{2}_{0} \}$ is complete and minimal in $L^{p}(\mathbb{T}^{2})$.
\end{example}

\subsection{The case $\Omega^{c} = 0\times\mathbb{Z}$}

\

Further in this section it is supposed that $\Omega \subset \mathbb{Z}^{2}$ is such that $\Omega^{c} = 0\times \mathbb{Z}$.%
\begin{lemma}\label{lem3.1}%
Let $g\in S_{\Omega }$ then $g(x,y) = h(y)$, where $h\in L^{1}(\mathbb{T})$.
\end{lemma}
\begin{proof}%
Let
\[
h(y) = \frac{1}{2\pi}\int_{\mathbb{T}} g(x,y) dx.
\]
Then $h\in L^{1}(\mathbb{T})$ and  for any $k\in \mathbb{Z}$
\[
c_{m}(h) = \frac{1}{2\pi}\int_{\mathbb{T}} h(y) e^{-i my} dy = \frac{1}{(2\pi)^{2}}\int_{\mathbb{T}^{2}} g(x,y)e^{-i my} dx dy.
\]
It is easy to check that for any $(k,m)\in \mathbb{Z}^{2}$
\[
\frac{1}{(2\pi)^{2}}\int_{\mathbb{T}^{2}} [g(x,y) -h(y)] e^{-i kx}e^{-i my} dx dy = 0.
\]
\end{proof}

 \begin{definition}\label{df2:1}
 Let $M\in L^{p}(\mathbb{T}^{2})$ and $1\leq p < \infty$.  We say that the   function $M(x,y)$ has a strong $x-$singular\-ity of
 degree $p\, (1\leq p <\infty)$   if for any  measurable set $E\subset \mathbb{T}, |E|>0$
 \[
 M^{-1}\notin L^{p'}(\mathbb{T}\times E), \qquad \frac{1}{p} + \frac{1}{p'} =1.
 \]
\end{definition}

\begin{proposition}\label{pr2:C}
Let $1\leq p <\infty$ and  let $M\in L^{p}(\mathbb{T}^{2})$.   Then the system
\begin{equation}\label{eq:sys0}
\{ M(x,y) e^{i kx}e^{i my}: (k,m)\in \Omega \}
\end{equation}
is complete in $L^{p}(\mathbb{T}^{2})$ if and only if  $M(x,y)$ has  a strong $x-$singularity of
 degree $p$.
\end{proposition}
\begin{proof}
Suppose that the   function $M(x,y)$ has  a strong $x-$singularity of
 degree $p\, (1\leq p <\infty)$.    If for some $g\in S_{\Omega}$ we have that $\frac{g}{\overline{M}} \in L^{p'}(\mathbb{T}^{2})$ then by Lemma \ref{lem3.1}  it follows that
\begin{equation}\label{eq:CC}
 \int_{\mathbb{T}^{2}}  |g(x,y)|^{p'}  |M(x,y)|^{-p'} dx dy = \int_{\mathbb{T}} |h(y)|^{p'}\int_{\mathbb{T}} |M(x,y)|^{-p'} dx dy <+\infty.
\end{equation}
Hence, the set $G=\{y\in \mathbb{T}: |h(y)|> 0\}$ should be of measure zero. Which yields that
$g(y) =0$ a.e. on $\mathbb{T}$ and by Theorem \ref{thm2:C} follows that the system (\ref{eq:sys0}) is complete in $ L^{p}(\mathbb{T}^{2})$.

For the proof of the necessity suppose that the system (\ref{eq:sys0}) is complete in $ L^{p}(\mathbb{T}^{2})$.  Hence, by Theorem \ref{thm2:C} we have that for any non trivial $g\in S_{\Omega}$
\[
 \int_{\mathbb{T}^{2}} \frac{|g(y)|^{p'}}{|M(x,y)|^{p'}} dx dy = +\infty.
\]
For any measurable set $E\subset \mathbb{T}, |E|>0$ we have that $\chi_{E}(y) \in S_{\Omega}$ which yields that $M(x,y)$ has  a strong $x-$singularity of
 degree $p$.
\end{proof}
For our further study we define a class of functions $\Upsilon$.

 \begin{definition}\label{df2:2}
  We say that $\phi \in \Upsilon$ if  $ \phi \in L^{\infty}(\mathbb{T})$ and $ \frac{1}{\phi} \in L^{\infty}(\mathbb{T})$.
\end{definition}

\begin{definition}\label{df2:3}
    We say that a   function $M\in L^{p}(\mathbb{T}^{2})$ has an $(x,P)-$sin\-gula\-rity of
 degree $p\, (1\leq p <\infty)$   if
 $M^{-1}\notin L^{p'}(\mathbb{T}^{2})$ and
 \begin{equation}\label{eq:xP}
 \int_{\mathbb{T}^{2}} \frac{|e^{ix} -P(y)|^{p'}}{|M(x,y)|^{p'}} dx dy < +\infty,
 \end{equation}
 where $P\in \Upsilon$.
\end{definition}

\begin{proposition}\label{pr2:M}
Let $1\leq p <\infty$ and  let $M\in L^{p}(\mathbb{T}^{2})$.  Then the system (\ref{eq:sys0})
is minimal in $L^{p}(\mathbb{T}^{2})$ if and only if one of the following conditions hold:
\begin{equation}\label{eq:pe}
 \int_{\mathbb{T}^{2}} |M(x,y)|^{-p'} dx dy < +\infty
\end{equation}
or the  function $M(x,y)$ has an $(x,P)-$singularity of
 degree $p$.
\end{proposition}
\begin{proof}
At first we suppose that (\ref{eq:pe}) holds.
Let
\[
\psi_{j,l}(x,y) = (2\pi )^{-2}  e^{i jx}e^{i ly} [\overline{ M(x,y)}]^{-1} \quad \mbox{for}\quad (j,l)\in \Omega.
\]
One can easily check that the system $\{ \psi_{j,l}(x,y): (j,l)\in \Omega\}\subset L^{p'}(\mathbb{T}^{2})$ is biorthogonal with (\ref{eq:sys0}).

 Now let us suppose that the  function $M(x,y)$ has an $(x,P)-$ singularity of
 degree $p$. Let
\begin{equation}\label{eq:bis}
\xi_{k,m}(x,y) = (2\pi)^{-2} [e^{i kx}e^{i my}- P^{k}(y) e^{i my}] [ \overline{M(x,y)}]^{-1}, \qquad (k,m)\in \Omega.
\end{equation}
Clearly $\xi_{k,m} \in L^{p'}(\mathbb{T}^{2})$ for any $(k,m)\in \Omega$. Moreover, it is easily that the system $\{ \xi_{k,m}(x,y): (k,m)\in \Omega\}$ is biorthogonal with (\ref{eq:sys0}).

Suppose that the system (\ref{eq:sys0})
is minimal in $L^{p}(\mathbb{T}^{2})$. Then by Theorem \ref{thm2:M} we have that the system $\{\varphi_{j,l} \}_{(j,l)\in \Omega} $ biorthogonal with (\ref{eq:sys0}) is defined by the equations (\ref{eq:bisy})
and $g_{k,m}\in S_{\Omega}$. If $g_{0,1}(x,y) = 0$ a.e. then (\ref{eq:pe}) holds.
If $\frac{1}{M} \notin L^{p'}(\mathbb{T}^{2})$ then $g_{0,1}$ is a non trivial  function and by Lemma \ref{lem3.1}  it we have that $g_{0,1}(x,y) = h_{0,1}(y)$. Let
\begin{eqnarray*}
P(y)  =   - h_{1,0}(y) \  \ & \mbox{if} \  \  & \frac{1}{2}<|h_{1,0}(y)|<2;\\
 P(y)  =                     \frac{1}{2} \, &\mbox{if}\ \  &|h_{1,0}(y)|\leq \frac{1}{2}\  \   \mbox{and} \\
                    2 \ &\mbox{if}\  \   &|h_{1,0}(y)|\geq 2.
\end{eqnarray*}
Clearly $P\in \Upsilon$ and  by the relation
\[
\frac{ e^{i x} + g_{1,0}(y) }{(2\pi)^{2}\overline{M(x,y)}} \in L^{p'}(\mathbb{T}^{2}).
\]
 it is easy to check that $M(x,y)$ has an $(x,P)-$ singularity of
 degree $p$.
\end{proof}
\begin{definition}\label{df2:4}
    We say that  $M\in L^{p}(\mathbb{T}^{2})$ has a strong $(x,P)-$singu\-lar\-ity of
 degree $p\, (1\leq p <\infty)$   if $M$ has a strong $x-$singularity  and an $(x,P)-$sin\-gu\-lar\-ity of
 degree $p $ for some $P\in \Upsilon$.
\end{definition}
\begin{proposition}\label{pr3:CM}
Let $1\leq p <\infty$ and  let $M\in L^{p}(\mathbb{T}^{2})$.  Suppose that  $\Omega^{c} =0\times \mathbb{Z} $. Then the system (\ref{eq:sys0})
is complete and minimal in $L^{p}(\mathbb{T}^{2})$ if and only if  the  function $M(x,y)$ has a strong  $(x,P)-$ singularity of
 degree $p$ with $|P(y)|\equiv 1$ a.e. on $\mathbb{T}$.
\end{proposition}
\begin{proof}
By Propositions \ref{pr2:C} and \ref{pr2:M} we have to show that if the system (\ref{eq:sys0})
is complete and minimal in $L^{p}(\mathbb{T}^{2})$ then the conditions of the proposition hold with $|P(y)|\equiv 1$ a.e. on $\mathbb{T}$. We provide the proof by
 reduction to absurdity. Suppose that $|P(y)| \neq 1$ if $y\in E, |E|>0$. Then for some $\delta >0$ we have that $| |P(y)| - 1| > \delta$ if $y\in F\subset E, |F|>0$. On the other hand we have that (\ref{eq:xP}) holds. Hence,
 \[
  \int_{\mathbb{T}\times F} \frac{1}{|M(x,y)|^{p'}} dx dy < +\infty
 \]
 which contradicts the condition that $M(x,y)$ has a strong $x-$singularity of
 degree $p$.
 The proof of sufficiency is obvious.
\end{proof}

\begin{lemma}\label{lem:bi}
Let $1\leq p <\infty$ and   $\Omega \subset \mathbb{Z}^{2}$ is such that $\Omega^{c} =0\times \mathbb{Z} $. Suppose that $M\in L^{p}(\mathbb{T}^{2})$ has a strong  $(x,P)-$singularity of
 degree $p$ with $|P(y)|\equiv 1$. Then the system (\ref{eq:sys0}) is  complete minimal in $L^{p}(\mathbb{T}^{2})$ and its conjugate   system $\{ \xi_{j,l}(x,y): (j,l)\in \Omega\}$
  is defined by the conditions (\ref{eq:bis}) and for any $(n,m)\in \mathbb{Z}^{2}$
\[
e^{inx} e^{imy}(e^{ix} -P(y)) =  \overline{M(x,y)} \xi_{n+1,m}(x,y) - P(y)\overline{M(x,y)} \xi_{n,m}(x,y).
\]
\end{lemma}
\begin{proof}
The first part of the lemma follows by Proposition \ref{pr3:CM} and the proof of Proposition \ref{pr2:M}.

 For any $(n,m)\in \mathbb{Z}^{2}$ we write
 \begin{eqnarray*}
 e^{inx} e^{imy}(e^{ix} -P(y))  &- \overline{M(x,y)} \xi_{n+1,m}(x,y) \\
  =  -P(y)e^{inx} e^{imy} + P^{n+1}(y) e^{i my} = &- \overline{M(x,y)} P(y) \xi_{n,m}(x,y).
\end{eqnarray*}
\end{proof}

\begin{thm}\label{thm:MB2}
Let $1\leq p <\infty$ and   $\Omega \subset \mathbb{Z}^{2}$ is such that $\Omega^{c} =0\times \mathbb{Z} $. Suppose that    is such that  Then the system  (\ref{eq:sys0}) is an $M-$basis in $L^{p}(\mathbb{T}^{2})$ if and only if $M\in L^{p}(\mathbb{T}^{2})$ has a strong  $(x,P)-$singularity of
 degree $p$ with $|P(y)|\equiv 1$.
\end{thm}
\begin{proof}
If the system  (\ref{eq:sys0}) is an $M-$basis in $L^{p}(\mathbb{T}^{2})$  then by  Proposition \ref{pr3:CM} it follows that $M$ has a strong  $(x,P)-$singularity of  degree $p$ with $|P(y)|\equiv 1$. On the other hand if
 the  function $M(x,y)$ has a strong  $(x,P)-$ singularity of degree $p$ with $|P(y)|\equiv 1$ a.e. on $\mathbb{T}$ then by  Proposition \ref{pr3:CM} the system  (\ref{eq:sys0}) is complete and minimal in $L^{p}(\mathbb{T}^{2})$ and
the system $\{ \xi_{j,l}(x,y): (j,l)\in \Omega\}$ conjugate to (\ref{eq:sys0}) is defined by the equations (\ref{eq:bis}).
 Let $f\in L^{p}(\mathbb{T}^{2})$ be such that
\[
\int_{\mathbb{T}^{2}} f(x,y)\overline{\xi_{j,l}}(x,y) dx dy =0, \quad \mbox{for all}\quad (j,l)\in \Omega.
\]
Then by (\ref{eq:sys0}) and the Fubini-Tonelli theorem we will have that
\[
\int_{\mathbb{T}} e^{-i ly}  \int_{\mathbb{T}} f(x,y)[e^{-i jx}- \overline{P^{j}(y)}][\overline{M(x,y)}]^{-1} dx dy =0, \quad \forall \, l\in \mathbb{Z}, j\in \mathbb{Z}_{0}.
\]
Which yields $\Phi_{j}(y) = 0$ a.e. on $\mathbb{T}$ for all $j\in \mathbb{Z}_{0},$ where
\[
\Phi_{j}(y) =  \int_{\mathbb{T}} f(x,y)[e^{-i jx}- \overline{P^{j}(y)}][\overline{M(x,y)}]^{-1} dx \qquad  j\in \mathbb{Z}_{0}.
\]

Let $y_{0}\in \mathbb{T}$ be such that the following conditions hold:
 \[
 \int_{\mathbb{T}} |f(x,y_{0})|^{p} dx<  +\infty,\ \Phi_{j}(y_{0}) = 0 \ \mbox{ for all} \ j\in \mathbb{Z}_{0},
 \]
  and
\[
 \int_{\mathbb{T}} \frac{1}{|M(x,y_{0})|^{p'}} dx  = +\infty,\quad  \int_{\mathbb{T}} \frac{|e^{ix} -P(y_{0})|^{p'}}{M(x,y_{0})|^{p'}} dx  < +\infty.
\]
Thus we have that $P(y_{0}) = e^{ix_{0}}$ for some $x_{0}\in \mathbb{T}$ and
\[
\int_{\mathbb{T}} f(x,y_{0})[e^{-i jx}- e^{-i jx_{0}}][\overline{M(x,y_{0})}]^{-1} dx =0 \qquad  \mbox{for all}\quad j\in \mathbb{Z}_{0},
\]
where $f(\cdot,y_{0}) \in L^{p}(\mathbb{T})$.   According   to the corresponding result in the one dimensional case (see \cite{KSK:3}) it follows that $ f(x,y_{0})= 0$ for almost any $x\in \mathbb{T}$. On the other hand we have that the above conditions  are true for almost all $y\in \mathbb{T}$. Which yields that
 $f=0$,  a.e. on $\mathbb{T}^{2}.$
\end{proof}

\begin{example}
Let
\[
M(x,y) = |\sin\frac{x-x_{0}}{2}|^{\alpha} \ \ \mbox{for} \ (x,y) \in \mathbb{T}^{2},
\]
 where $x_{0}\in \mathbb{T}$ and $\frac{1}{p'} \leq \alpha < 1+ \frac{1}{p'}$. It is easy to check that $m$ has
a strong  $(x,P)-$singularity of
 degree $p$ with $P(y)= e^{ix_{0}}$ if $y\in \mathbb{T}$. By Theorem \ref{thm:MB2} it follows that the system (\ref{eq:sys0}) is an $M-$basis in in $L^{p}(\mathbb{T}^{2})$ with the conjugate system
 \[
 \xi_{k,m}(x,y) = (2\pi)^{-2} [e^{i kx}e^{i my}-  e^{i kx_{0}}e^{i my}] [ \overline{M(x,y)}]^{-1}, \qquad (k,m)\in \Omega.
 \]
\end{example}

\subsection{The case $\Omega^{c} = 0\times\mathbb{Z}_{0}$}

\

In the cases studied above we have that if the system $\{  M(x,y) e^{i kx}e^{i my}: (k,m)\in \Omega \}$ is complete and minimal in $ L^{p}(\mathbb{T}^{2})$ then it is  an $M-$basis in $L^{p}(\mathbb{T}^{2})$. Suppose that $\Omega_{0} \subset \mathbb{Z}^{2}$ is such that $\Omega_{0}^{c} = 0\times \mathbb{Z}_{0}$. In this section we prove that if the system
\begin{equation}\label{eq:sys1}
\{  M(x,y) e^{i kx}e^{i my}: (k,m)\in \Omega_{0} \}
\end{equation}
is complete in $ L^{p}(\mathbb{T}^{2})$ then it is not minimal.

\begin{thm}\label{pr:C2}
Let $1\leq p <\infty$ and  let  $M\in L^{p}(\mathbb{T}^{2})$. Suppose that $\Omega_{0} \subset \mathbb{Z}^{2}$ is such that $\Omega_{0}^{c} = 0\times \mathbb{Z}_{0}$. Then the system (\ref{eq:sys1})
is complete in $L^{p}(\mathbb{T}^{2})$ if and only if the  function $M(x,y)$ has a strong $x-$ singularity of
 degree $p$.
\end{thm}
\begin{proof}
By Proposition \ref{pr2:C} it is clear that if
 the weight function $w(x,y)$ has a strong $x-$ singularity of
 degree $p $  then the system (\ref{eq:sys1}) is complete in $ L^{p}(\mathbb{T}^{2})$.

For the proof of the necessity suppose that the system (\ref{eq:sys1}) is complete in $ L^{p}(\mathbb{T}^{2})$. By Theorem \ref{thm2:C} we have that for any non trivial $g\in S_{\Omega_{0}}$
\[
\int_{\mathbb{T}^{2}} \frac{|g(y)|^{p'}}{|M(x,y)|^{p'}} dx dy = +\infty.
\]
Let $E\subset \mathbb{T}, |E|>0$ be any measurable set and $E_{1}\subset E$ be such that $|E_{1}| = \frac{1}{2} |E|.$ It is easy to observe that  we have that $\chi_{E_{1}}(y) - \chi_{E\setminus E_{1}}(y) \in S_{\Omega_{0}}$ which yields that $M(x,y)$ has a strong $x-$ singularity of
 degree $p$.
\end{proof}

\begin{proposition}\label{pr:CM2}
Let $1\leq p <\infty$ and  suppose that $\Omega_{0} \subset \mathbb{Z}^{2}$ is such that $\Omega_{0}^{c} =0\times \mathbb{Z}_{0}$. Then for any  function $M\in L^{p}(\mathbb{T}^{2})$ the system (\ref{eq:sys})
is not complete minimal in $L^{p}(\mathbb{T}^{2})$.
\end{proposition}
\begin{proof}
Suppose that for a  function $M(x,y)$ the system (\ref{eq:sys})
is  complete minimal in $L^{p}(\mathbb{T}^{2})$.
By Proposition \ref{pr:C2} we have that  $M(x,y)$ has a strong $x-$ singularity of
 degree $p$. By  Theorem \ref{thm2:M} it follows that there exists  $g_{0,0} \in S_{\Omega_{0}}$ such that
 \[
 \frac{1+ g_{0,0}(y)}{M(x,y)} \in L^{p'}(\mathbb{T}^{2}).
 \]
Which yields that
\[
 \int_{\mathbb{T}}|1+ g_{0,0}(y)|^{p'} \int_{\mathbb{T}} \frac{1}{|M(x,y)|^{p'}} dx dy < +\infty.
 \]
The last condition
  contradicts  the condition that $M(x,y)$ has a strong $x-$ singularity of
 degree $p$.
\end{proof}
We say that $g\in \Upsilon_{0}$ if $g\in \Upsilon$ and $\int_{\mathbb{T}} g(t) dt = 0.$
By Lemma \ref{lem3.1} it easily follows that if $g\in S_{\Omega_{0} }$ then $g(x,y) = h(y)$, where $h\in L^{1}(\mathbb{T})$ and $\int_{\mathbb{T}} h(y) dy = 0.$
\begin{proposition}\label{pr4:M}
Let $1\leq p <\infty$ and  let $M\in L^{p}(\mathbb{T}^{2})$.  Then the system (\ref{eq:sys})
is minimal in $L^{p}(\mathbb{T}^{2})$ if and only if or holds the condition (\ref{eq:pe})
or the  function $M(x,y)$ has an $(x,P)-$sin\-gularity of
 degree $p$ with $P\in \Upsilon_{0}$ and for some $Q\in \Upsilon_{0}$
\begin{equation}\label{eq:up0}
\int_{\mathbb{T}^{2}}|1- Q(y)|^{p'}|M(x,y)|^{-p'} dx dy < \infty.
\end{equation}
\end{proposition}
\begin{proof}
If the condition (\ref{eq:pe}) holds then the proof is similar to the proof of Proposition \ref{pr2:M}.

 Now let us suppose that the  function $M(x,y)$ has an $(x,P)-$ singularity of
 degree $p$ with $P\in \Upsilon_{0}$ and for some $Q\in \Upsilon_{0}$ holds the condition (\ref{eq:up0}).
 Clearly $ \Upsilon_{0} \subset S_{\Omega_{0} }$. Thus if we put
 \[
\xi_{0,0}(x,y) = (2\pi)^{-2} [1- Q(y) e^{i my}] [ \overline{M(x,y)}]^{-1},
 \]
and for $(k,m)\in \Omega_{0}\setminus (0,0)$
\[
\xi_{k,m}(x,y) = (2\pi)^{-2} [e^{i kx}e^{i my}- P^{k}(y) e^{i my}] [ \overline{M(x,y)}]^{-1}.
\]
Clearly $\xi_{k,m} \in L^{p'}(\mathbb{T}^{2})$ for any $(k,m)\in \Omega_{0}$. Moreover, it is easily that the system $\{ \xi_{k,m}(x,y): (k,m)\in \Omega_{0}\}$ is biorthogonal with (\ref{eq:sys1}).

Suppose that the system (\ref{eq:sys1})
is minimal in $L^{p}(\mathbb{T}^{2})$. Then by Theorem \ref{thm2:M} we have that the system $\{\varphi_{j,l} \}_{(j,l)\in \Omega_{0}} $ biorthogonal with (\ref{eq:sys1}) is defined by the equations (\ref{eq:bisy})
and $g_{k,m}\in S_{\Omega_{0}}$. If $g_{0,0}(y) = 0$ a.e. then (\ref{eq:pe}) holds.
If $\frac{1}{M} \notin L^{p'}(\mathbb{T}^{2})$ then $g_{0,0}$ is a non trivial  function and by Lemma \ref{lem3.1}  it we have that $g_{0,0}(x,y) = h_{0,0}(y)$ and $\int_{\mathbb{T}} h_{0,0}(y) dy = 0$. Let $Q(y) = h_{0,0}(y)$ if $y\in G_{0}:=\{y\in \mathbb{T}: \frac{1}{2}<|h_{1,0}(y)|<2 \}$ and let $\beta = \int_{G_{0}} h_{1,0}(y) dy$. If $|G_{0}| = 2\pi$ then the function $Q$ is defined and $Q\in \Upsilon_{0}$. If $|G_{0}| < 2\pi$ then putting $Q(y) = \frac{-\beta}{2\pi - |G_{1}| }$ if $ y\in \mathbb{T}\setminus G_{1}$. Clearly $Q\in \Upsilon_{0}$and the relation (\ref{eq:up0}) holds. In a similar way we define $P\in \Upsilon_{0}$
so that
\[
\frac{ e^{i x} + P(y) }{(2\pi)^{2}\overline{M(x,y)}} \in L^{p'}(\mathbb{T}^{2}).
\]
Thus $M(x,y)$ has an $(x,P)-$ singularity of
 degree $p$.

\end{proof}
%

%
%

\end{document}